\documentclass[12pt]{article}
\usepackage[utf8]{inputenc}
\usepackage[english]{babel}
\usepackage{amsfonts}
\usepackage{amsmath, amssymb}
\usepackage{amsthm}
\usepackage{cite}
\usepackage[left=2cm,right=2cm,    top=2cm,bottom=2cm,bindingoffset=0cm]{geometry}

\def\ind{\operatorname{ind}}

\def\Td{\operatorname{Td}}
\def\ch{\operatorname{ch}}

\def\Ell{\operatorname{Ell}}
\def\Mat{\operatorname{Mat}}
\DeclareMathOperator\tr{Tr}

\def\dim{\operatorname{dim}}

\def\det{\operatorname{det}\nolimits}

\DeclareMathOperator\Con{Con}

\newtheorem{theorem}{Theorem}
\newtheorem{lemma}{Lemma}
\newtheorem{prop}{Proposition}

\theoremstyle{definition}
\newtheorem{definition}{Definition}
\title{Nonlocal elliptic problems  associated with  actions of discrete groups on manifolds with boundary}
\author{Boltachev~A.V., Savin~A.Yu.}
\date{}

\begin{document}
\clearpage
\maketitle

\begin{abstract}
Given a manifold with boundary endowed with an action of a discrete group on it, we consider the algebra of operators generated by elements in the Boutet de Monvel algebra of pseudodifferential boundary value problems and shift operators acting on functions on the manifold and  its boundary. Under certain conditions on the group action, we construct a Chern character for elliptic elements in this algebra with values in a de Rham type cohomology of the fixed point manifolds for te group action and obtain an index formula in terms of this Chern character.
\end{abstract}

\section{Introduction}
The aim of this work is to solve the index problem for nonlocal elliptic boundary value problems associated with   actions of discrete groups on manifolds with boundary.

The index problem in  elliptic theory (see~\cite{AtSi1}) consists in computing indices of  elliptic operators in terms of topological invariants of the principal symbol of   operators and topological invariants of  manifolds on which the operators are defined.

Index formulas are known in many geometric situations. In particular, many  authors developed index theory for elliptic boundary value problems in the framework of classical boundary value problems (see~\cite{AtBo2, Hor3, WRL1}) and for Boutet de Monvel algebra of pseudodifferential boundary value problems with boundary and coboundary operators (see~\cite{Bout2, Esk1, ReSc1, Fds17, MeNeSc1, MeScSc1}).

Firstly, we note that obtaining index formulas in the case of pseudodifferential boundary value problems was of significant difficulty because principal symbols of the boundary value problems in question are operator functions defined on the cotangent bundle of the boundary of the manifold and it was necessary to construct topological invariants which take into account this operator function. These difficulties were overcame in~\cite{Fds17}, where a	 topological index of elliptic problems in Boutet de Monvel algebra was defined using special regularized traces on the algebra of principal symbols.

Secondly, we note that all known proofs of the index formulas for elliptic boundary value problems use stable homotopies (these homotopies are constructed using $K$-theory of the $C^*$-algebra of symbols), which allow to reduce the boundary value problem to such a form that its index is equal to the index of an operator on the double of the manifold with boundary. In this case,  the index is easy to calculate using Atiyah-Singer formula (see~\cite{Hor3, ReSc1, Fds17, MeScSc1}).

Noncommutative geometry of Connes~\cite{Con1} contributes greately to the development of  index theory. In noncommutative geometry, one usually considers algebras of operators, the principal symbols of which generate essentially noncommutative algebras. The algebras of crossed product type associated with  group actions on manifolds (see~\cite{Con4, CoDu1, AnLe1, AnLe2, LaSu1, CoLa2, PoWa2, Per3, NaSaSt17, Ros6, SaSt30}) appear in many applications. The corresponding equations on the manifold in question include pseudodifferential operators as well as shift operators along the orbits of the group action. For such operators, ellipticity conditions were obtained which provide Fredholm solvability of the problem in Sobolev spaces (see~\cite{AnLe2}). However, index formulas were obtained only in the case of smooth closed manifolds (see~\cite{NaSaSt17, SaSt30, Per3, Per5}). As a first step towards obtaining   index formulas for boundary value problems associated with group actions on manifolds with boundary, a classification (up to stable homotopies) of   elliptic boundary value problems was obtained and $K$-groups of the corresponding symbol algebra were calculated in~\cite{SaSt43}.

In the present paper, we construct the topological index for elliptic boundary value problems on manifolds with boundary  endowed with an isometric action of a discrete group of polynomial growth in the sense of Gromov~\cite{Gro1}. More precisely, for the principal symbol of such a problem we construct the Chern character with values in a de Rham type cohomology of the cotangent bundle of fixed point submanifolds. The definition of Chern character uses traces on the algebra of noncommutative differential forms from~\cite{NaSaSt17} and on regularized traces from~\cite{Fds17} on the algebra of symbols of Boutet de Monvel operators. Unlike the paper~\cite{Fds17}, where a topological index is constructed as a number, we improve this construction and construct a Chern character for the symbol in suitable cohomology groups, and the topological index is obtained from it by taking product with the Todd class of the manifold and integrating over the fundamental cycle.

Let us briefly describe the contents of the paper. In Sec.~2, we recall the definitions related to the Boutet de Monvel algebra of pseudodifferential boundary value problems on a manifold with boundary. In Sec.~3, we study the algebra generated by Boutet de Monvel operators and shift operators associated with isometric actions of discrete groups of polynomial growth on manifolds with boundary. We state an ellipticity condition for the elements in this algebra, their Fredholm property is established. After that we solve the index problem. To this end, we firstly (in Sec.~4) construct cohomology of de Rham type for the manifolds, whose boundary is the total space of a fibration (the cotangent bundle of a manifold with boundary has this structure and this structure  implicitly appears in~\cite{Fds17}). Furthermore, in Sec.~5, we introduce noncommutative differential forms on the boundary and regularized traces on them. These constructions enable us to define the Chern character for elliptic problems as a cohomology class on the cotangent bundle of the manifold of fixed points of the group action, considered as a manifold with fibered boundary (in the sense of Sec.~4). Next, in Sec.~6, we give the definition of the Todd class of the manifold and state the index theorem. The proof of the index theorem is given in Secs.~7 and~8. To this end, we establish a homotopy classification of elliptic boundary value problems, which enables us to make a homotopy of the boundary value problem to a quite simple boundary value problem in a neighborhood of the boundary of the manifold. Next, for such a problem   the index formula can be checked by a reduction of the operator to the double and application of the index formula in~\cite{NaSaSt17} on the double.

The work of the first author was supported by RFBR grant 19-01-00574, the work of the second author was supported by RFBR grant 19-01-00447.

\section{Boutet de Monvel algebra}
Let us recall the main properties of the Boutet de Monvel algebra (for more detailed exposition see~\cite{Bout2, ReSc1, Gru3, Schr3}).

Let $M$ be a smooth compact manifold with boundary $X$. We suppose that $M$ is endowed with a Riemannian metric. In a neighborhood of the boundary we will use local coordinates $x=(x',x_n)$ on $M$, where ${\rm dim}M=n$, $x'=(x_1,\ldots,x_{n-1})$ are the local coordinates on $X$, and $x_n$ is a defining function of the boundary, i.e., the boundary is locally defined by the equation $x_n=0$, while $M$ is defined by the inequality $x_n\geq 0$. Let us consider Boutet de Monvel operators of zero order and type. We write such operators as follows
\begin{equation}
\label{eq-opb1}
\mathcal{D}= \left(
 \begin{array}{cc}
   A+G & C\\ 
   B & A_X 
 \end{array}
\right):
\begin{array}{c}
   L^2(M)\\ 
   \oplus \\
   L^2(X)
 \end{array}
 \longrightarrow
 \begin{array}{c}
   L^2(M)\\ 
   \oplus \\
   L^2(X)
 \end{array},
\end{equation}
where
\begin{itemize}
\item $A$ is a pseudodifferential operator ($\psi$DO) of order zero on $M$, the principal symbol of which satisfies the transmission property (see below);
\item $A_X$ is a $\psi$DO of order zero on $X$;
\item $B,C$ and $G$ are boundary, coboundary and Green operators, respectively.
\end{itemize}
Recall that a classical symbol $a=a(x',x_n,\xi',\xi_n)$ with an asymptotic expansion $a~a_m+a_{m-1}+...$ satisfies the {\em transmission property}, if its order $m\in\mathbb{Z}$ and for all $k\in\mathbb{Z}_+$ and arbitrary multi-index $\alpha=(\alpha_1,...,\alpha_{n-1})\in\mathbb{Z}^{n-1}_+$ the following equality is satisfied for its homogeneous components $a_l$:
$$
D^k_{x_n}D^\alpha_{\xi'}a_l(x',0,0,1)=e^{i\pi(l-|\alpha|)}D^k_{x_n}D^\alpha_{\xi'}a_l(x',0,0,-1),
$$
where
$$
D^\alpha_{\xi'}=\left(-i\frac{\partial }{\partial \xi_1}\right)^{\alpha_1}\cdot ... \cdot \left(-i\frac{\partial }{\partial \xi_{n-1}}\right)^{\alpha_{n-1}}.
$$
The principal symbol of operator \eqref{eq-opb1} is the pair $\sigma(\mathcal{D})=(\sigma_M(\mathcal{D}),\sigma_X(\mathcal{D}))$, where the first component is called the {\em interior symbol} and is a function
\begin{equation}\label{s-bdm0}
 \sigma_M(\mathcal{D})=\sigma(A)\in C^\infty(S^*M),
\end{equation}
where $S^*M=\{(x,\xi)\in T^*M,|\xi|=1\}$ is the cosphere bundle of $M$, where $\sigma(A)$  is the principal symbol of a $\psi$DO $A$. 
The second component is called the {\em boundary symbol} and is an operator function
$$
\sigma_X(\mathcal{D}) \in 
 C^\infty(S^*X,\mathcal{B}(L^2(\mathbb{R}_+)\oplus \mathbb{C}))
\simeq 
 C^\infty(S^*X,\mathcal{B}(\overline H_+\oplus \mathbb{C}))
$$
on the cosphere bundle $S^*X$ of the boundary, where $\mathcal{B}$ is the algebra of bounded operators,
$$
H_+=\mathcal{F}(\mathcal{S}(\overline{\mathbb{R}}_+))
$$
stands for the space of images with the respect to the Fourier transform $\mathcal{F}_{x_n\to \xi_n}$ of the Schwartz space $\mathcal{S}(\overline{\mathbb{R}}_+)$ of smooth rapidly decaying at infinity functions on   $\overline{\mathbb{R}}_+$. Denote the norm closure of  $H_+\subset L^2(\mathbb{R})$ by $\overline{H}_+$.  Similarly, we define the space $H_-$ as
$$
H_-=\mathcal{F}(\mathcal{S}(\overline{\mathbb{R}}_-)).
$$
Next, we define the projection $\Pi_+:H_+\oplus H_-\to H_+$ to the first summand and the continuous functional
$$
\begin{array}{ccc}
\Pi':H_+\oplus H_- & \longrightarrow & \mathbb{C},\vspace{2mm}\\
u(\xi_n) & \longmapsto &\lim_{x_n\to 0+} \mathcal{F}^{-1}_{\xi_n\to x_n}(u(\xi_n)).
\end{array}
$$
Note that if $u\in L^1(\mathbb{R})\cap(H_+\oplus H_-)$, then
\begin{equation}
\label{Pi'u}
\Pi'u=\frac{1}{2\pi}\int_\mathbb{R}u(\xi_n)d\xi_n.
\end{equation}
Let us denote the algebra of {\em boundary symbols} $\sigma_X(\mathcal{D})$ by  $\Sigma_X\subset C^\infty(S^*X,\mathcal{B}(\overline H_+\oplus \mathbb{C}))$. 
To describe the elements of the algebra $\Sigma_X$, we consider smooth functions
\begin{itemize}
\item $b(x',\xi',\xi_n)\in C^\infty(S^*X,H_-)$;
\item $c(x',\xi',\xi_n)\in C^\infty(S^*X,H_+)$;
\item $g(x',\xi',\xi_n,\eta_n)\in C^\infty(S^*X,H_+\otimes H_-)$;
\item $q(x',\xi')\in C^\infty(S^*X)$.
\end{itemize}
Here the spaces $H_\pm$ and their topological tensor products are considered as Fr\'echet spaces. An arbitrary boundary symbol $a_X\in \Sigma_X$ is a smooth operator family
\begin{equation}
\label{eq-spb1}
a_X (x',\xi')= \left(
 \begin{array}{cc}
   \Pi_+a(x',0,\xi',\xi_n)+\Pi'_{\eta_n}g(x',\xi',\xi_n,\eta_n) & c(x',\xi',\xi_n)\\[2mm]
   \Pi'_{\xi_n}b(x',\xi',\xi_n) &  q(x',\xi') 
 \end{array}
\right):
\begin{array}{c}
   \overline{H}_+\\ 
   \oplus \\
   \mathbb{C}
 \end{array}
 \longrightarrow
 \begin{array}{c}
   \overline{H}_+\\ 
   \oplus \\
   \mathbb{C}
 \end{array}
\end{equation}
with the parameters $(x',\xi')\in S^*X$. Here $a(x',0,\xi',\xi_n)$ is the restriction  to the boundary of a symbol homogeneous of order zero with the transmission property on $M$. The function $a(x',0,\xi',\xi_n)$ is called the {\em principal symbol} of the boundary symbol $a_X$. The operator of type \eqref{eq-spb1} acts on the pair $h\in \overline H_+, v\in\mathbb{C}$ as follows:
\begin{equation}
\label{a_X_act}
a_X(x',\xi')
\left(
\begin{matrix}
h \\
v
\end{matrix}
\right)
=
\left(
\begin{matrix}
\Pi_+(a(x',0,\xi',\xi_n)h(\xi_n))+\Pi'_{\eta_n}(g(x',\xi',\xi_n,\eta_n)h(\eta_n)) + c(x',\xi',\xi_n)v \\[2mm]
\Pi'_{\xi_n}(b(x',\xi',\xi_n)h(\xi_n)) + q(x',\xi')v
\end{matrix}
\right)
\end{equation}
It is known that   smooth families in \eqref{eq-spb1} and \eqref{a_X_act} form an algebra. 

Let us denote the algebra of matrices  \eqref{eq-opb1} by $\Psi_B(M)\subset \mathcal{B}(L^2(M)\oplus L^2(X))$. This algebra is called the {\em Boutet de Monvel algebra}.
\begin{theorem}
(\cite{ReSc1}, Section 2.2.4.4, Corollary 2) The symbol mapping
$$
\begin{array}{ccc}
\Psi_B(M) & \longrightarrow & C^\infty(S^*M)\oplus C^\infty(S^*X,\mathcal{B}(L^2(\mathbb{R}_+)\oplus \mathbb{C}))\\
\mathcal{D} & \longmapsto &(\sigma_M(\mathcal{D}),\sigma_X(\mathcal{D}))
\end{array}
$$
is well defined and continuously extends to a monomorphism of $C^*$-algebras
$$
\overline{\Psi_B(M)} \longrightarrow  C(S^*M)\oplus C(S^*X,\mathcal{B}(L^2(\mathbb{R}_+)\oplus \mathbb{C})).
$$
\end{theorem}

The regularized trace $\tr'$ of a boundary symbol$a_X\in\Sigma_X$ is defined by the formula
\begin{equation}
\label{reg_trace}
\tr'a_X(x',\xi') = \Pi'_{\eta_n} g(x',\xi',\eta_n,\eta_n) + q(x',\xi').
\end{equation}
The mapping $\tr'$  does not possess the  trace property. More precisely, the following trace defect formula was obtained in   \cite[Section 2.4, Lemma 2.1]{Fds17}:
given   $a_{X,1}, a_{X,2}\in \Sigma_X$, we have  
\begin{equation}
\label{FedLem0}
\tr'[a_{X,1}, a_{X,2}] = -i\Pi'\left(\frac{\partial a_1(\xi_n)}{\partial\xi_n}a_2(\xi_n)\right) = 
i\Pi'\left(a_1(\xi_n)\frac{\partial a_2(\xi_n)}{\partial\xi_n}\right), 
\end{equation}
where $a_1, a_2$ are the principal symbols of   $a_{X,1}, a_{X,2}$ respectively.

\section{$\Gamma$-Boutet de Monvel operators. Fredholm property}

\paragraph{Group actions and shift operators.}
Let $\Gamma$ be a discrete finitely generated group of isometries $\gamma\colon M\to M$, which preserve the boundary $\gamma(X)=X$. Given $\gamma\in \Gamma$, we define the shift operator
$$
T_\gamma:L^2(M) \oplus L^2(X)\longrightarrow L^2(M) \oplus L^2(X),\quad (u(x),v(x'))\longmapsto (u(\gamma^{-1}(x)),v(\gamma^{-1}(x'))).
$$
This operator is unitary if we equip the $L^2$-spaces with the norms, defined by the volume forms  associated with the Riemannian metric. The mapping
$\gamma\mapsto T_\gamma$ defines a unitary representation of $\Gamma$ on $L^2(M) \oplus L^2(X)$.
 
It is known that compositions $T_\gamma\mathcal{D}T_\gamma^{-1}$, where $\mathcal{D}$ is a Boutet de Monvel operator and $\gamma\in \Gamma$, is also a Boutet de Monvel operator. Moreover, the interior and boundary symbols of $T_\gamma\mathcal{D}T_\gamma^{-1}$ are equal to
$$
\sigma_M(T_\gamma\mathcal{D}T_\gamma^{-1})(x,\xi)=\sigma_M(\mathcal{D})(\partial\gamma^{-1}(x,\xi)),
\quad \sigma_X(T_\gamma\mathcal{D}T_\gamma^{-1})(x',\xi')=\sigma_X(\mathcal{D})(\partial\gamma^{-1}(x',\xi')).
$$
Here the actions of $\Gamma$ on $M$ and $X$ are lifted to the bundles ${T}^*M$ and $T^*X$ using codifferentials $\partial \gamma=(d\gamma^t)^{-1}$
of the corresponding diffemorphisms (here $d\gamma$ is the differential of $\gamma$, while $d\gamma^t$ is its dual mapping of the cotangent bundle).

\paragraph{$\Gamma$-Boutet de Monvel operators.}
Let us recall the definition of smooth crossed products (see \cite{Schwe1} or \cite{SaSt30}). Let $\mathcal{A}$ be a Fr\'echet algebra with seminorms $\|\cdot\|_m$, $m\in\mathbb{N}$, and $\Gamma$ be a group of polynomial growth in the sense of Gromov (see \cite{Gro1}), acting on $\mathcal{A}$ by automorphisms $a\mapsto \gamma(a)$, where $a\in \mathcal{A}$ and $\gamma\in\Gamma$. Then the {\em smooth crossed product} denoted by $\mathcal{A}\rtimes \Gamma$ is the vector space of functions $f:\Gamma\to \mathcal{A}$, which rapidly decay at infinity in the sence of the following estimates:
$$
\|f(\gamma)\|_m\le C_{m,N}(1+|\gamma|)^{-N}
$$
for all $N,m\in\mathbb{N}$ and $\gamma\in\Gamma$, where the constant $C_{m,N}$ does not depend on $\gamma$. Here $|\gamma|$ is the length $\gamma$ in the word metric on $\Gamma$. Finally, we assume that the action of $\Gamma$ on $\mathcal{A}$ satisfies the following property: given $m\in\mathbb{N}$, there exists $k\in\mathbb{N}$ and a real polynomial $P(z)$  such that the following inequality
$$
\|\gamma(a)\|_m\le P(|\gamma|)\|a\|_k
$$ 
holds for all $a$ and $\gamma$. 
The product in $\mathcal{A}\rtimes \Gamma$ is defined by the formula: 
\begin{equation}
\label{mult}
\{f_1(\gamma)\}\cdot \{f_2(\gamma)\}=\left\{\sum_{\gamma_1\gamma_2=\gamma}f_1(\gamma_1)\gamma_1(f_2(\gamma_2))\right\}.
\end{equation}
It can be shown that the right hand side in \eqref{mult} is an element of $\mathcal{A}\rtimes \Gamma$, i.e., this space is an algebra. 

The action of $\Gamma$ on Fr\'echet algebras $C^\infty(S^*M), \Sigma_X$ and $\Psi_B(M)$ satisfies the conditions above, and the following smooth crossed products are defined: $C^\infty(S^*M)\rtimes \Gamma, \Sigma_X\rtimes \Gamma$ and $\Psi_B(M)\rtimes \Gamma$. Elements in the smooth crossed product $\Psi_B(M)\rtimes \Gamma$ define operators  
\begin{equation}
\label{eq-op1-d}
\left\{\mathcal{D}_\gamma\right\}_{\gamma\in\Gamma}\in\Psi_B(M)\rtimes \Gamma\quad\longmapsto\quad\sum_{\gamma\in \Gamma} \mathcal{D}_\gamma T_\gamma: L^2(M)\oplus L^2(X)\to L^2(M)\oplus L^2(X) 
\end{equation}

\begin{definition}
Operators in \eqref{eq-op1-d} are called $\Gamma$-{\em Boutet de Monvel operators}.
\end{definition}

\begin{definition}
The {\em symbol} of   operator \eqref{eq-op1-d} is a pair $\sigma(\mathcal{D})=(\sigma_M(\mathcal{D}),\sigma_X(\mathcal{D}))$, which consists of the interior and the boundary symbols
\begin{equation}\label{s-bdm1}
 \sigma_M(\mathcal{D})=\{\sigma(A_\gamma)\}_{\gamma\in\Gamma}\in C^\infty(S^*M)\rtimes \Gamma, \quad \sigma_X(\mathcal{D})=\{\sigma_X(\mathcal{D}_\gamma)\}_{\gamma\in\Gamma}\in \Sigma_X\rtimes \Gamma.
\end{equation}
\end{definition}
The operators \eqref{eq-op1-d} form an algebra and symbols \eqref{s-bdm1} enjoy the composition formula: given $\mathcal{D}_1,\mathcal{D}_2\in\Psi_B(M)\rtimes \Gamma$ we have $\sigma_M(\mathcal{D}_1\mathcal{D}_2)=\sigma_M(\mathcal{D}_1)\sigma_M(\mathcal{D}_2)$ and $\sigma_X(\mathcal{D}_1 \mathcal{D}_2)=\sigma_X(\mathcal{D}_1)\sigma_X(\mathcal{D}_2)$.

\paragraph{Operators, acting between ranges of projections.}
A matrix $\Gamma$-{\em Boutet de Monvel operator} is a triple $(\mathcal{D}, \mathcal{P}_1, \mathcal{P}_2)$ such that
\begin{itemize}
\item $ \mathcal{D}\in{\rm Mat}_N(\Psi_B(M)\rtimes \Gamma)$ 
is a matrix   with the components in $\Psi_B(M)\rtimes \Gamma $;
\item  
$ \mathcal{P}_j\in{\rm Mat}_N(C^\infty(M)\rtimes \Gamma\oplus C^\infty(X)\rtimes \Gamma),\; j=1,2; $ 
\item the following relation is satisfied $\mathcal{P}_2\mathcal{D}\mathcal{P}_1=\mathcal{D}.$
\end{itemize}
Then we define the operator
\begin{equation}\label{eq-2a}
\mathcal{D}:
\mathcal{P}_1(L^2(M,\mathbb{C}^N)\oplus L^2(X,\mathbb{C}^N))
\longrightarrow 
\mathcal{P}_2(L^2(M,\mathbb{C}^N)\oplus L^2(X,\mathbb{C}^N)).
\end{equation}
Denote operator \eqref{eq-2a} by $(\mathcal{D}, \mathcal{P}_1, \mathcal{P}_2)$. 

We also introduce the notation
$$
P_j=\sigma_M(\mathcal{P}_j),\qquad P'_j=\sigma_X(\mathcal{P}_j).
$$

\begin{definition}
\label{def_elliptic}
Operator $(\mathcal{D}, \mathcal{P}_1, \mathcal{P}_2)$ is   {\em elliptic}, if there exists a matrix operator $(\mathcal{R}, \mathcal{P}_2, \mathcal{P}_1)$ such that the following equalities hold
$$
\sigma(\mathcal{D})\sigma(\mathcal{R})=\sigma(\mathcal{P}_2),\qquad\sigma(\mathcal{R})\sigma(\mathcal{D})=\sigma(\mathcal{P}_1).
$$
\end{definition}

\begin{theorem}
An elliptic operator \eqref{eq-2a} has Fredholm property.
\end{theorem} 
The proof is standard (see, for instance,~\cite{AnLe2}).

\section{Two de Rham complexes for manifolds with fibered boundary}

Let us consider a Riemannian manifold $M$ with boundary $\partial M$ and suppose that the boundary is a total space of a locally trivial fiber bundle $\pi:\partial M\to X$ with the fiber $F$. Then the pair $(M,\pi)$ is called a {\em manifold with fibered boundary}. 

The embedding $i:\partial M\to M$ induces the restriction mapping $i^*:\Omega^*(M)\to \Omega^*(\partial M)$ of the differential forms to the boundary. The projection $\pi$ defines the induced mapping $\pi^*:\Omega^*(X)\to \Omega^*(\partial M)$ and the mapping
$$
\pi_*:\Omega_c^*(\partial M)\longrightarrow \Omega_c^{*-\nu}(X),\quad \nu=\dim F,
$$
of integration of compactly supported differential forms over the fiber $F$ (see, for instance,~\cite{BGV1}). Here we suppose that the fibers of $\pi$ have orientation  continuously depending on the point of the base. Let us recall the definition of the integral of form over the fiber.
\begin{definition}
The integral of a form $\omega\in\Omega_c^k(\partial M)$ over the fibers of $\pi:\partial M\to X$ is a differential form $\pi_*{\omega}\in\Omega_c^{k-\nu}(X)$  such that 
$$
\int\limits_{X}{\Big(\pi_*{\omega}\Big)\wedge\omega_1} = \int_{\partial M}{\omega}\wedge\pi^*\omega_1
$$
for all differential forms $\omega_1$ on $X$. 
The following properties are valid:
\begin{enumerate}
\item $\pi_*(\omega\wedge\pi^*\omega_1)=(\pi_*\omega)\wedge\omega_1,$ for all forms $\omega\in\Omega_c^k(\partial M), \omega_1\in\Omega_c^k(X)$;
\item $d(\pi_*\omega) = (-1)^\nu\pi_*(d\omega)$, for all forms $\omega\in\Omega_c^k(\partial M)$.
\end{enumerate} 
\end{definition}

Let us consider the graded morphism
\begin{equation}
\label{eq-grmor1}
(\Omega_c^*(M),d) \stackrel{\pi_*i^*}\longrightarrow (\Omega_c^{*-\nu}(X),d),\quad d\pi_*i^*=(-1)^\nu\pi_*i^*d 
\end{equation}
of the de Rham complexes on $M$ and $X$. Denote the cone of $\pi_*i^*$  by $(\Omega^*(M,\pi),\partial)$, where
\begin{equation}\label{eq-c1}
 \Omega_c^j(M,\pi)=\Omega_c^j(M)\oplus\Omega_c^{j-\nu-1}(X), \quad
 \partial=\left(
            \begin{array}{cc}
             d & 0\\
             \pi_*i^*& (-1)^{\nu+1} d
            \end{array}\right).
\end{equation}
Note that
\begin{multline*}
 \partial^2=\left(
            \begin{array}{cc}
             d & 0\\
             \pi_*i^*& (-1)^{\nu+1} d
            \end{array}\right)
						\left(
            \begin{array}{cc}
             d & 0\\
             \pi_*i^*& (-1)^{\nu+1} d
            \end{array}\right) =
						\left(
            \begin{array}{cc}
             d^2 & 0\\
             \pi_*i^*d+(-1)^{\nu+1} d(\pi_*i^*)& (-1)^{2\nu+2} d^2
            \end{array}\right) 
\\
						=\left(
            \begin{array}{cc}
             0 & 0\\
             \pi_*i^*d+(-1)^{\nu+1} d\pi_*i^*& 0
            \end{array}\right)
						=
						0.
\end{multline*}
The last equality follows from~\eqref{eq-grmor1}.

We denote the cohomology groups of the complex $(\Omega_c^*(M,\pi),\partial)$ by $H_c^* (M,\pi)$.

Also, we consider the complex $(\widetilde{\Omega}^*(M,\pi),\widetilde\partial)$:
\begin{equation}\label{eq-c2}
 \widetilde{\Omega}^j(M,\pi)=\{(\omega,\omega_X)\in\Omega^j(M)\oplus\Omega^j(X)\;|\; i^*\omega=\pi^*\omega_X\}, \quad
 \widetilde{\partial}=\left(
            \begin{array}{cc}
             d & 0\\
             0& d
            \end{array}\right).  
\end{equation}
Note that $di^*\omega=d\pi^*\omega_X$ since $i^*\omega=\pi^*\omega_X$. Hence, $i^*d\omega=\pi^*d\omega_X$, since $i^*$ and $\pi^*$ are the induced mappings on differential forms, i.e., the mapping $\widetilde{\partial}$ is well defined.
We denote the cohomology groups of the complex $(\widetilde{\Omega}^*(M,\pi),\widetilde{\partial})$ by $\widetilde{H}^*(M,\pi)$.

Component-wise exterior multiplication of differential forms gives the product
$$
\wedge: \Omega_c^j(M,\pi)\times \widetilde{\Omega}^k(M,\pi)\longrightarrow \Omega_c^{j+k}(M,\pi).
$$
This operation enjoys  the Leibniz rule
\begin{equation}
\label{eq-leibniz1}
\partial (a\wedge b)=\partial a\wedge b+(-1)^{j}a\wedge \widetilde\partial b,\quad a\in \Omega_c^j(M,\pi), b\in \widetilde{\Omega}^k(M,\pi).
\end{equation}
Indeed, given 
$$
a=
						\left(
            \begin{array}{c}
             \omega\\
             \omega_X
            \end{array}\right) \quad\text{and}\quad
						b=
						\left(
            \begin{array}{c}
             \omega'\\
             \omega'_X
            \end{array}\right),
$$
we have
\begin{multline*}
\partial (a\wedge b)=\partial\left(
            \begin{array}{c}
             \omega\wedge\omega'\\
             \omega_X\wedge\omega'_X
            \end{array}\right)=
					 \left(
            \begin{array}{c}
             d(\omega\wedge\omega')\\
             \pi_*i^*(\omega\wedge\omega')+(-1)^{\nu+1}d(\omega_X\wedge\omega'_X)
            \end{array}\right)
            \\
		   =
						\left(
            \begin{array}{c}
             d\omega\wedge\omega'+(-1)^j\omega\wedge d\omega'\\
             \pi_*(i^*\omega\wedge i^*\omega')+(-1)^{\nu+1}d\omega_X\wedge\omega'_X+(-1)^{\nu+1}(-1)^{j-\nu-1}\omega_X\wedge d\omega'_X
            \end{array}\right)
						\\
						=
						\left(
            \begin{array}{c}
             d\omega\wedge\omega'\\
             \pi_*i^*(\omega)\wedge\omega'_X+(-1)^{\nu+1}d\omega_X\wedge\omega'_X
            \end{array}\right)
						+
						\left(
            \begin{array}{c}
             (-1)^j\omega\wedge d\omega'\\
             (-1)^j\omega_X\wedge d\omega'_X
            \end{array}\right)
						\\
					 =\partial\left(
            \begin{array}{c}
             \omega\\
             \omega_X
            \end{array}\right)\wedge
						\left(
            \begin{array}{c}
             \omega'\\
             \omega'_X
            \end{array}\right)
						+
						(-1)^j\left(
            \begin{array}{c}
             \omega\\
             \omega_X
            \end{array}\right)\wedge
						\widetilde\partial
						\left(
            \begin{array}{c}
             \omega'\\
             \omega'_X
            \end{array}\right).
\end{multline*}
It follows from the Leibniz rule~\eqref{eq-leibniz1}
 that   $\wedge$ defines a product in cohomology
\begin{equation}
\label{cohom_prod}
\wedge:H_c^j(M,\pi)\times \widetilde{H}_c^k(M,\pi)\longrightarrow H_c^{j+k}(M,\pi).
\end{equation}

Finally, we suppose that $M$ and $X$ are oriented manifolds and their orientations are compatible with the orientation of the fibers in the following way. 
Denote $n={\rm dim}M$ and choose as positively oriented the form
$(-1)^n dt_1\wedge\ldots\wedge dt_\nu\wedge dy_1\wedge\ldots\wedge dy_k\wedge dx_n$, where $t_1,\ldots,t_\nu$ are some positively oriented coordinates on the fiber, $y_1,\ldots,y_{n-\nu-1}$ are some positively oriented coordinates on $X$,
while $x_n\ge 0$ is a defining function of the boundary. 

We define the linear functional
$$
\begin{array}{rcc}
  \displaystyle\langle \cdot,[M,\pi]\rangle:H_c^* (M,\pi) & \longrightarrow & \mathbb{C}\vspace{2mm}\\
   (\omega,\omega_X) & \longmapsto &  \displaystyle\int_M \omega- (-1)^n\int_X \omega_X.
\end{array}
$$ 
To prove that this functional is well defined, we choose forms
  $\omega=f(t,y,x_n)dt_1\wedge\ldots \wedge dy_{n-\nu-1}$ and $\omega_X=g(y)dy_1\wedge \ldots \wedge dy_{n-\nu-1}$ and compute
\begin{multline*}
\langle \partial(\omega,\omega_X),[M,\pi]\rangle=
\int\limits_M d\omega-(-1)^n\int\limits_X \left(\pi_*i^*\omega+(-1)^{\nu+1}d\omega_X\right) 
=\int\limits_{\partial M} i^*\omega-(-1)^n\int\limits_X \left(\pi_*i^*\omega\right)
\\
=(-1)^n\!\!\!\!\int\limits_{\mathbb{R}^\nu\times \mathbb{R}^{n-\nu-1}}\!\!\!\!\!\!\!f(t,y,0)dt_1...dy_{n-\nu-1}-(-1)^n\!\!\!\!\int\limits_{\mathbb{R}^{n-\nu-1}}{\left(\int\limits_{\mathbb{R}^\nu}{f(t,y,0)dt_1...dt_\nu}\right)dy_1...dy_{n-\nu-1}}=0.
\end{multline*}

\section{Chern character of elliptic symbols}
\paragraph{Noncommutative differential forms. Regularized trace.}

In this section, we define the Chern character for symbols of elliptic operators \eqref{eq-2a}. The definition uses noncommutative differential forms which we now define. Denote by
$$
\widetilde\Sigma_X\subset C^\infty(T^*X,\mathcal{B}(\overline H_+\oplus \mathbb{C}))
$$ 
the algebra of boundary symbols on $T^*X$. 


Let us consider the action of $\Gamma$ on the algebras $\Omega(T^*M)$ and $\widetilde\Sigma_X\otimes_{C^\infty(X)}\Omega(T^*X)$ of differential forms with compact support and corresponding smooth crossed products
$$
\Omega(T^*M)\rtimes \Gamma \text{ and }\left(\widetilde\Sigma_X\otimes_{C^\infty(X)}\Omega(T^*X)\right)\rtimes \Gamma.
$$
These crossed products are differential graded algebras with respect to the gradings defined by the gradings of the differential forms and the differentials  defined by the exterior differential on differential forms.

Given $\gamma\in \Gamma$, we define the mappings (cf. \cite{NaSaSt17,SaSt26})
\begin{equation}\label{eq-tr1}
 \tau^{\gamma}:\Omega(T^*M)\rtimes \Gamma \longrightarrow \Omega(T^*M^{\gamma}),
\end{equation}
\begin{equation}\label{eq-tr2}
 \tau^\gamma_{X}:\left(\widetilde\Sigma_X\otimes_{C^\infty(X)}\Omega(T^*X)\right)\rtimes \Gamma \longrightarrow \Omega(T^*X^{\gamma}).
\end{equation}
To define these mappings, we introduce some notations. We denote the closure of $\Gamma$ in the compact Lie group of isometries of $M$ by $\overline{\Gamma}$. This closure is a compact Lie group. 
Let $C^\gamma\subset\overline{\Gamma}$ be the centralizer\footnote{We recall that the centralizer of $\gamma$ is a subgroup of elements which commute with $\gamma$.}
of $\gamma$. The centralizer is a closed Lie subgroup in $\overline{\Gamma}$. The elements of centralizer are denoted by $h$, while the induced Haar measure on the centralizer is denoted as $dh$. Next, for an element $\gamma'\in \langle
\gamma\rangle$ in the conjugacy class of $\gamma$, we fix an arbitrary element $z=z(\gamma,\gamma')$, which conjugates $\gamma$ and
$\gamma'=z\gamma z^{-1}$. Any such element defines the diffeomorphism $\partial z:T^*M^\gamma\to T^*M^{\gamma'}$ for the corresponding fixed point manifolds.

We define the functional in \eqref{eq-tr1} by
\begin{equation}\label{eq-sled-nash1}
    \tau^\gamma(\omega)=
      \sum_{\gamma'\in \langle \gamma\rangle}\;\;\;
        \int_{C^\gamma}
           \Bigl.h^*\bigl(
              {z}^*\omega(\gamma')
             \bigr)\Bigr|_{T^*M^\gamma}
           dh,\quad \text{where }\omega\in \Omega(T^*M)\rtimes \Gamma,
\end{equation}
and the functional \eqref{eq-tr2} by
\begin{equation}\label{eq-sled-nash2}
    \tau^\gamma_X(\omega_X)=
      \sum_{\gamma'\in \langle \gamma\rangle}\;\;\;
        \int_{C^\gamma} \tr_X
				\left(
           \Bigl.h^*\bigl(
              {z}^*\omega_X(\gamma')
             \bigr)\bigr|_{T^*X^\gamma}
         \right) 
					dh,\quad \text{where }\omega_X\in \left(\widetilde\Sigma_X\otimes_{C^\infty(X)}\Omega(T^*X)\right)\rtimes \Gamma.
\end{equation}
Here 
$$
\tr_X \left(
\sum_I \omega_I(t)dt^I\right)= \sum_I \tr'(\omega_I(t))dt^I,
$$
where $\tr':\widetilde\Sigma_X \to C^\infty(T^*X)$ is the regularized trace defined earlier in \eqref{reg_trace}.

\begin{prop} The following assertions hold:
\begin{enumerate}
\item The summands in \eqref{eq-sled-nash1} and \eqref{eq-sled-nash2} do not depend on the choice of the elements $z$.
\item Functionals \eqref{eq-sled-nash1} and \eqref{eq-sled-nash2} have the following properties:
\begin{equation}\label{eq-trcl1}
\tau^\gamma(\omega_1\wedge\omega_2)=(-1)^{\deg\omega_1\deg\omega_2}\tau^\gamma(\omega_2\wedge\omega_1),\text{ for all } \omega_1,\omega_2\in \Omega(T^*M)\rtimes\Gamma,
\end{equation}
\begin{equation}\label{eq-trcl2}
d\tau^\gamma_X(\omega )= \tau^\gamma_X(d\omega ),\quad \text{for all }\omega \in \left(\widetilde\Sigma_X\otimes_{C^\infty(X)}\Omega(T^*X)\right)\rtimes \Gamma.
\end{equation}
\end{enumerate}
\end{prop}
\begin{proof}
Let us prove 1. Indeed, let $z_1$ be another element such that $\gamma'=z_1\gamma z_1^{-1}$. Then for the functional \eqref{eq-sled-nash1} we have:
$$
        \int_{C^\gamma}
           \Bigl.h^*\bigl(
              z_1^*\omega(\gamma')
             \bigr)\Bigr|_{T^*M^\gamma}
           dh=
        \int_{C^\gamma}
           \Bigl.(z_1h)^*\bigl(
              \omega(\gamma')
             \bigr)\Bigr|_{T^*M^\gamma}
           dh=
					 \int_{C^\gamma}
           \Bigl.h^*\bigl(z^*
              \omega(\gamma')
             \bigr)\Bigr|_{T^*M^\gamma}
           dh.
$$
Here in the last equality we made the change of variable $z_1h=zh'$ in the integral and used the invariance of the Haar measure. Equality \eqref{eq-sled-nash2} is proved similarly.

Now let us move on to part 2. 
Firstly, we show that $\tau^\gamma$ is a graded trace, i.e., equality \eqref{eq-trcl1} holds. It suffices to prove this property for the following forms $\omega_1, \omega_2$:
$$
\omega_1(\gamma)=\left\{
\begin{matrix}
a, &\gamma=\gamma_1,\\
0, &\gamma\ne\gamma_1,
\end{matrix}
\right.
\qquad
\omega_2(\gamma)=\left\{
\begin{matrix}
b, &\gamma=\gamma_2,\\
0, &\gamma\ne\gamma_2.
\end{matrix}
\right.
$$
Then
$$
(\omega_1\wedge\omega_2)(\gamma)=\left\{
\begin{matrix}
0, &\gamma\ne\gamma_1\gamma_2,\\
a\wedge{\gamma_1^*}^{-1}b, &\gamma=\gamma_1\gamma_2.
\end{matrix}
\right.
$$
Since functional \eqref{eq-sled-nash1} does not depend on the choice of $z$, we can consider $z=e\in\Gamma$. 

On the one hand, for $\gamma=\gamma_1\gamma_2$ we have
\begin{equation}
\label{pr_indep_1}
\tau^\gamma(\omega_1\wedge\omega_2)=
        \int_{C^\gamma}\Bigl.h^*\Bigl(
				a\wedge{\gamma_1^*}^{-1}b\Bigr)\Bigr|_{T^*M^\gamma}dh=
\int_{C^\gamma}\Bigl.(h^*
				a\wedge h^*{\gamma_1^*}^{-1}b)\Bigr|_{T^*M^\gamma}dh=
				\int_{C^\gamma}\Bigl.(h^*
				a\wedge \gamma^*h^*{\gamma_1^*}^{-1}b)\Bigr|_{T^*M^\gamma}dh,
\end{equation}
since the form is integrated over ${T^*M^\gamma}$ and over this manifold we have $\gamma^*h^*=h^*$. 
Since $h\in C^\gamma$ in \eqref{pr_indep_1}, we have $h^*\gamma^*=\gamma^*h^*$. We then transform \eqref{pr_indep_1} as
\begin{equation}
\label{pr_indep_2}
\tau^\gamma(\omega_1\wedge\omega_2)=
\int_{C^\gamma}\Bigl.h^*
				a\wedge h^*\gamma^*{\gamma_1^*}^{-1}b\Bigr|_{T^*M^\gamma}dh=
\int_{C^\gamma}\Bigl.h^*\Bigl(
				a\wedge \gamma_2^*b\Bigr)\Bigr|_{T^*M^\gamma}dh,
\end{equation}
since $\gamma=\gamma_1\gamma_2$ and $\gamma^*{\gamma_1^*}^{-1}=\gamma_2^*$. 

On the other hand, for $z=\gamma_2$, we have:
$$
\tau^\gamma(\omega_2\wedge\omega_1)=
\int_{C^\gamma}\Bigl.h^*\Bigl(z^*\Bigl(b\wedge{\gamma_2^*}^{-1}a\Bigr)\Bigr)\Bigr|_{T^*M^\gamma}dh=
(-1)^{{\rm deg}\omega_1{\rm deg}\omega_2}\int_{C^\gamma}\Bigl.h^*\Bigl(z^*\Bigl(
				{\gamma_2^*}^{-1}a\wedge b\Bigr)\Bigr)\Bigr|_{T^*M^\gamma}dh
$$
$$
=(-1)^{{\rm deg}\omega_1{\rm deg}\omega_2}\int_{C^\gamma}\Bigl.h^*\Bigl(
				a\wedge \gamma_2^*b\Bigr)\Bigr|_{T^*M^\gamma}dh=
(-1)^{{\rm deg}\omega_1{\rm deg}\omega_2}\tau^\gamma(\omega_1\wedge\omega_2).
$$
Here in the last equality we used \eqref{pr_indep_2}. Now we prove equality \eqref{eq-trcl2}:
\begin{multline*}
\tau^\gamma_X(d\omega_X)=
      \sum_{\gamma'\in \langle \gamma\rangle}\;
        \int_{C^\gamma} \tr_X
				\left(
           \Bigl.h^*\bigl(
              {z}^*(d\omega_X(\gamma'))
             \bigr)\bigr|_{T^*X^\gamma}
        \right) 
					dh
\\
				=\sum_{\gamma'\in \langle \gamma\rangle}\;
        \int_{C^\gamma} \tr_X
				\left(
           \Bigl.d\bigl(h^*\bigl(
              {z}^*(\omega_X(\gamma'))
             \bigr)\bigr)\bigr|_{T^*X^\gamma}
        \right) 
					dh=
				d\sum_{\gamma'\in \langle \gamma\rangle}\;
        \int_{C^\gamma} \tr_X
				\left(
           \Bigl.h^*\bigl(
              {z}^*(\omega_X(\gamma'))
             \bigr)\bigr|_{T^*X^\gamma}
         \right) 
					dh
\\
					=d\tau^\gamma_X(\omega_X).
\end{multline*}

\end{proof}

\paragraph{Definition of the Chern character.}
Consider an elliptic operator $(\mathcal{D}, \mathcal{P}_1, \mathcal{P}_2)$. For brevity, we frequently denote it by $\mathcal{D}$. We extend the interior symbols $\sigma_M(\mathcal{D}), \sigma_M(\mathcal{R})$ of the original operator and its almost inverse to $T^*M$ as smooth symbols, which have the transmission property. We extend the boundary symbols
$\sigma_X(\mathcal{D})$ and $\sigma_X(\mathcal{R})$ to $T^*X$ as smooth symbols. We denote such extensions by
$$
a,r\in C^\infty(T^*M)\rtimes \Gamma, \qquad a_X,r_X\in   \widetilde\Sigma_X  \rtimes \Gamma.
$$
Suppose that these extensions are compatible, i.e., the principal symbol of the boundary symbol is equal to the restriction of the interior symbol to the boundary and the following equalities hold:
$$
a=P_2aP_1,\quad r=P_1rP_2,\quad a_X=P'_2a_XP'_1,\quad r_X=P'_1r_XP'_2.
$$
We define the noncommutative connections
$$
\nabla_{P_j}=P_j\cdot d\cdot P_j \text{ on }T^*M\quad \text{ and }\quad
\nabla_{P'_j}=P'_j\cdot d'\cdot P'_j \text{ on }T^*X, \text{ where } j=1,2,
$$
and also the connection on $T^*M$
\begin{equation}
\label{nabla_M}
\widetilde{\nabla}_{P_1}=\nabla_{P_1}+r\nabla a, \text{ where } \nabla a \equiv \nabla_{P_2}a-a\nabla_{P_1}.
\end{equation}
Similarly, we define the connection on $T^*X$
$$
\widetilde{\nabla}_{P'_1}=\nabla_{P'_1}+r_X\nabla' a_X, \text{ where } \nabla' a_X =\nabla_{P'_2}a_X-a_X\nabla_{P'_1}.
$$ 
\begin{lemma}
The curvatures of connections $\widetilde{\nabla}_{P_1}$ and $\widetilde{\nabla}_{P'_1}$ are equal to
$$
\widetilde{\Omega}_{P_1}\overset{\mathrm{def}}{=}(\widetilde{\nabla}_{P_1})^2=\nabla_{P_1}^2+\nabla_{P_1}r\nabla a + (r\nabla a)^2,
$$
$$
\widetilde{\Omega}_{P'_1}\overset{\mathrm{def}}{=}(\widetilde{\nabla}_{P'_1})^2=\nabla_{P'_1}^2+\nabla_{P'_1}r_X\nabla' a_X + (r_X\nabla' a_X)^2.
$$
\end{lemma}
\begin{proof}
Indeed, we have
$$
\widetilde{\Omega}_{P_1}u=(\widetilde{\nabla}_{P_1})^2u=(\nabla_{P_1}+r\nabla a)^2u=
(\nabla_{P_1}^2+\nabla_{P_1}r\nabla a + r(\nabla a)\nabla_{P_1} + (r\nabla a)^2)u
$$
$$
=(\nabla_{P_1}^2+\nabla_{P_1}r\nabla a + (r\nabla a)^2)u,
$$
$$
\widetilde{\Omega}_{P'_1}v=(\widetilde{\nabla}_{P'_1})^2v=(\nabla_{P'_1}+r_X\nabla' a_X)^2v=
(\nabla_{P'_1}^2+\nabla_{P'_1}r_X\nabla' a_X + r_X(\nabla' a_X)\nabla_{P'_1} + (r_X\nabla' a_X)^2)v
$$
$$
=(\nabla_{P'_1}^2+\nabla_{P'_1}r_X\nabla' a_X + (r_X\nabla' a_X)^2)v.
$$
\end{proof}
Let us define the differential forms with compact supports
\begin{equation}\label{eq-forms1}
 \ch_{T^*M}^\gamma\sigma(\mathcal{D}) \in \Omega^{ev}_c(T^*M^\gamma),\quad \ch^\gamma_{T^*X}\sigma(\mathcal{D})\in \Omega^{ev}_c(T^*X^\gamma)
\end{equation}
on the cotangent bundles of the submanifolds of fixed points by the formulas
\begin{equation}
\label{form_ch_M}
\ch^\gamma_{T^*M}\sigma(\mathcal{D})=\tau^{\gamma}\left(e^{-\widetilde{\Omega}_{P_1}/2\pi i}(P_1-ra)\right)- 
\tau^{\gamma}\left(P_2e^{-\nabla^2_{P_2}/2\pi i}-ae^{-\widetilde{\Omega}_{P_1}/2\pi i}r\right)
\end{equation}
\begin{equation}
\label{form_ch_X}
\ch^\gamma_{T^*X}\sigma(\mathcal{D})=\tau^\gamma_X\left(e^{-\widetilde{\Omega}_{P'_1}/2\pi i}(P'_1-r_Xa_X)\right)- 
\tau^\gamma_X\left(P'_2e^{-\nabla^2_{P'_2}/2\pi i}-a_Xe^{-\widetilde{\Omega}_{P'_1}/2\pi i}r_X\right).
\end{equation}
Here and below we denote the extensions of mappings \eqref{eq-tr1} and \eqref{eq-tr2} to the matrix algebras over the  corresponding crossed products again by $\tau^\gamma, \tau_X^\gamma$. The extensions are obtained as the compositions of the matrix trace and mappings \eqref{eq-tr1} and \eqref{eq-tr2}. 
Since $\tau^\gamma$ is a graded trace, \eqref{form_ch_M} can be written as
$$
\ch^\gamma_{T^*M}\sigma(\mathcal{D})=\tau^{\gamma}\left(e^{-\widetilde{\Omega}_{P_1}/2\pi i}P_1 - 
P_2e^{-\nabla^2_{P_2}/2\pi i}\right).
$$

The boundary $\partial(T^*M^\gamma)\simeq T^*X^\gamma\times\mathbb{R}$ is fibered over $T^*X^\gamma$ with the fiber $\mathbb{R}$.
We denote the corresponding projection by $\pi^\gamma:\partial(T^*M^\gamma)\to T^*X^\gamma$ and the embedding
$\partial (T^*M^\gamma)\subset T^*M^\gamma$ by $i_\gamma$. Hence, the pair $(T^*M^\gamma,\pi^\gamma)$ is a manifold with fibered boundary in the sense of Sec. 4.
\begin{prop}
Given $\gamma\in\Gamma$, we have  
\begin{equation}
\label{eq-fund1}
d\left(\ch_{T^*M}^\gamma\sigma(\mathcal{D})\right)=0,
\end{equation}
\begin{equation}
\label{eq-fund1_1}
d'\left(\ch_{T^*X}^\gamma\sigma(\mathcal{D})\right)=\pi^\gamma_*i_\gamma^*\left(\ch_{T^*M}^\gamma\sigma(\mathcal{D})\right).
\end{equation}
In other words, the pair   $(\ch_{T^*M}^\gamma\sigma(\mathcal{D}),-\ch_{T^*X}^\gamma\sigma(\mathcal{D}))$  is closed in the complex $(\Omega_c^*(T^*M^\gamma,\pi^\gamma),\partial)$, see \eqref{eq-c1}, and we denote its cohomology class  by
$$
 \ch^\gamma\sigma(\mathcal{D}) \in H^{ev} (T^*M^\gamma,\pi^\gamma).
$$
This class does not depend on the choice of the elements $a,r,a_X,r_X$ and does not change under homotopies of elliptic symbols.
\end{prop}
\begin{proof}
1. Equality \eqref{eq-fund1} can be proven in a standard way (see, for instance, \cite{Zha3})
\begin{multline*}
d\left(\ch_{T^*M}^\gamma\sigma(\mathcal{D})\right) =  
d\tau^{\gamma}\left(e^{-\widetilde{\nabla}^2_{P_1}/2\pi i}P_1\right) - d\tau^{\gamma}\left(P_2e^{-{\nabla^2_{P_2}}/2\pi i}\right) 
\\
=\tau^{\gamma}\left(\widetilde{\nabla}_{P_1}\left(e^{-\widetilde{\nabla}^2_{P_1}/2\pi i}P_1\right)\right) - \tau^{\gamma}\left(\nabla_{P_2}\left(P_2e^{-{\nabla^2_{P_2}}/2\pi i}\right)\right) 
\\
=\tau^{\gamma}\left[\widetilde{\nabla}_{P_1},e^{-\widetilde{\nabla}^2_{P_1}/2\pi i}P_1\right] - \tau^{\gamma}\left[\nabla_{P_2},P_2e^{-{\nabla^2_{P_2}}/2\pi i}\right] = 0.
\end{multline*}
The last equality holds, since the commutators used in it are equal to zero.

2. Next, let us prove equality \eqref{eq-fund1_1}. First, we calculate the right hand side in \eqref{eq-fund1_1}. We denote the restriction of the curvature form $\widetilde{\Omega}_{P_1}$ to $\partial T^*M$ by $\Omega_1=\widetilde{\Omega}_{P_1}|_{\partial T^*M}$. It is clear that $\Omega_1$ is equal to the curvature form for the pair of restrictions $(a|_{\partial T^*M},r|_{\partial T^*M})$. Also we denote by $\Omega_2=\Omega_2(\xi_n)=\widetilde{\Omega}_{P_1}|_{\partial T^*X}$  the family of curvature forms for the restrictions $(a|_{\partial T^*M\cap\{\xi_n={\rm const}\}},r|_{\partial T^*M\cap\{\xi_n={\rm const}\}})$, where $\xi_n$ is considered as a parameter. We have
$$
i^*_\gamma\ch^\gamma_{T^*M}\sigma(\mathcal{D}) = \tau^\gamma\left(e^{-\Omega_1/2\pi i}P_1 - P_2e^{-{\nabla^2_{P_2}}/2\pi i}\right),
$$
\begin{equation}
\label{right_part_1_2}
\pi^\gamma_*i^*_\gamma\ch^\gamma_{T^*M}\sigma(\mathcal{D})=-\frac{1}{2\pi i}\int_{\mathbb{R}}{\tau^\gamma\left(\left(\frac{\partial}{\partial\xi_n}\lrcorner\Omega_1\right)
e^{-\Omega_2/2\pi i}\right)d\xi_n}.
\end{equation}
Here $\frac{\partial}{\partial\xi_n}\lrcorner\Omega$ is the substitution of the vector field into the differential form. By Lemma 1, we have
\begin{equation}
\label{substit}
\frac{\partial}{\partial\xi_n}\lrcorner\Omega_1 = \frac{\partial}{\partial\xi_n}\lrcorner
\left(\nabla_{P_1}^2+\nabla_{P_1}(r\nabla a) + (r\nabla a)^2\right).
\end{equation}
We now substitute $\partial/\partial \xi_n$ into each of the summands in \eqref{substit}. To this end, we represent the connections in the following form:
$$
d = d\xi_n\frac{\partial}{\partial\xi_n} +d',\quad \nabla_{P_j}=P_jd'P_j + P_jd\xi_n\frac{\partial}{\partial\xi_n}\equiv 
\nabla'_{P_j} + P_jd\xi_n\frac{\partial}{\partial\xi_n}, \nabla=\nabla' + d\xi_n\frac{\partial}{\partial\xi_n}.
$$
For the first summand in \eqref{substit}, we have
\begin{equation}
\label{first_term}
\frac{\partial}{\partial\xi_n}\lrcorner\nabla_{P_1}^2 = \frac{\partial}{\partial\xi_n}\lrcorner(P_1dP_1)^2=0,
\end{equation}
since $P_1$ does not depend on $\xi_n$. 
For the second summand in \eqref{substit}, we have
\begin{multline}
\label{sec_term_exp}
\nabla_{P_1}(r\nabla a) = \left(\nabla'_{P_1} + P_1d\xi_n\frac{\partial}{\partial\xi_n}\right)
\left(r\nabla'a + r\frac{\partial a}{\partial\xi_n}d\xi_n\right) 
\\
=\nabla'_{P_1}(r\nabla'a) + \nabla'_{P_1}\left(r\frac{\partial a}{\partial\xi_n}d\xi_n\right) + 
\left(P_1d\xi_n\frac{\partial}{\partial\xi_n}\right)(r\nabla'a) + \left(P_1d\xi_n\frac{\partial}{\partial\xi_n}\right)
\left(r\frac{\partial a}{\partial\xi_n}d\xi_n\right) 
\\
=\nabla'_{P_1}(r\nabla'a) + \nabla'_{P_1}\left(r\frac{\partial a}{\partial\xi_n}\right)d\xi_n + 
d\xi_n\frac{\partial}{\partial\xi_n}(r\nabla'a).
\end{multline}
We now substitute $\partial/\partial \xi_n$ into  \eqref{sec_term_exp}:
\begin{multline}
\label{second_term}
\frac{\partial}{\partial\xi_n}\lrcorner \nabla_{P_1}(r\nabla a) = 
-\nabla'_{P_1}\left(r\frac{\partial a}{\partial\xi_n}\right)+ 
\frac{\partial}{\partial\xi_n}(r\nabla'a) 
\\
=-\left(\nabla'r\right)\frac{\partial a}{\partial\xi_n}
-r\nabla'\frac{\partial a}{\partial\xi_n}+ 
\frac{\partial r}{\partial\xi_n}\nabla'a +
r\nabla'\frac{\partial a}{\partial\xi_n}=
-\left(\nabla'r\right)\frac{\partial a}{\partial\xi_n}+ 
\frac{\partial r}{\partial\xi_n}\nabla'a.
\end{multline}
For the third summand in \eqref{substit}, we get
\begin{equation}
\label{third_term}
\frac{\partial}{\partial\xi_n}\lrcorner \left(\left(r\nabla'a + r\frac{\partial a}{\partial\xi_n}d\xi_n\right)\left(r\nabla'a + r\frac{\partial a}{\partial\xi_n}d\xi_n\right)\right)=
r\frac{\partial a}{\partial\xi_n}r\nabla'a- (r\nabla'a)r\frac{\partial a}{\partial\xi_n}.
\end{equation}
Substituting  \eqref{first_term}, \eqref{second_term}, \eqref{third_term} into \eqref{right_part_1_2}, we obtain
\begin{equation}
\label{right_part_1_1}
\pi^\gamma_*i^*_\gamma\ch^\gamma_{T^*M}\sigma(\mathcal{D})=\frac{i}{2\pi}\int_{\mathbb{R}}{\tau^\gamma\left(\left(
\frac{\partial r}{\partial\xi_n}\nabla'a
-\left(\nabla'r\right)\frac{\partial a}{\partial\xi_n}+ 
\left[r\frac{\partial a}{\partial\xi_n},r\nabla'a\right]
\right)e^{-\Omega_2/2\pi i}\right)d\xi_n}.
\end{equation}

3. Now we calculate the left hand side in \eqref{eq-fund1_1}. 

\begin{lemma}
\label{omega_P1}
For each form $\omega'\in \Mat_N((\tilde{\Sigma}_X\otimes_{C^\infty(X)}\Omega(T^*X) )\rtimes \Gamma)$ such that $\omega'=P'_1\omega'P'_1$, we have
\begin{equation}
\label{omega'_lemma}
d'\tau^\gamma_X(\omega')=\tau^\gamma_X(\nabla_{P'_1}\omega').
\end{equation}
\end{lemma}
\begin{proof}
The difference between the left and right hand sides in \eqref{omega'_lemma} is equal to 
\begin{multline}
\label{omega'_lemma_proof}
d'\tau^\gamma_X(\omega')-\tau^\gamma_X(\nabla_{P'_1}\omega')=\tau^\gamma_X(d'(P'_1\omega')-P'_1d'\omega')=\tau^\gamma_X((d'P'_1)\omega') 
\\
=\tau^\gamma_X((d'P'_1)P'_1\omega'P'_1)=\tau^\gamma_X(P'_1d'P'_1P'_1\omega')=0,
\end{multline}
where in the last line we used the cyclic property: $\tau^\gamma_X:\tau^\gamma_X(\omega'P'_1)=\tau^\gamma_X(P'_1\omega')$. The last equality holds, since $P'_1$ acts as a scalar operator in the variable $\xi_n$. In the last equality in \eqref{omega'_lemma_proof} we used the identity $P'_1(d'P'_1)P'_1=0$ for the projection $P'_1$. 

This completes the proof of the Lemma. 
\end{proof}

Using Lemma~\ref{omega_P1}, we obtain the following expression for the left hand side in \eqref{eq-fund1_1}
\begin{equation}
\label{d'_expansion}
d'\ch^\gamma_{T^*X}\sigma(\mathcal{D})=\tau^\gamma_X\left(\nabla_{P'_1}\big(e^{-\widetilde{\nabla}^2_{P'_1}/2\pi i}P'_1\big)\right)-
\tau^\gamma_X\left(\nabla_{P'_2}\big(P'_2e^{-\nabla^2_{P'_2}/2\pi i}\big)\right)+
d'\tau^\gamma_X\left[a_X, e^{-\widetilde{\nabla}^2_{P'_1}/2\pi i}r_X\right].
\end{equation}
For the first summand in \eqref{d'_expansion}, we have
\begin{multline*}
\nabla_{P'_1}\left(e^{-\widetilde{\nabla}^2_{P'_1}/2\pi i}P'_1\right) = 
\big[\nabla_{P'_1},e^{-\widetilde{\nabla}^2_{P'_1}/2\pi i}P'_1\big] = 
\big[(\nabla_{P'_1}+r_X\nabla'a_X)-r_X\nabla'a_X,e^{-\widetilde{\nabla}^2_{P'_1}/2\pi i}P'_1\big]
\\
=\big[\widetilde{\nabla}_{P'_1},e^{-\widetilde{\nabla}^2_{P'_1}/2\pi i}P'_1\big]
-
\big[r_X\nabla'a_X,e^{-\widetilde{\nabla}^2_{P'_1}/2\pi i}P'_1\big]=
-
\big[r_X\nabla'a_X,e^{-\widetilde{\nabla}^2_{P'_1}/2\pi i}P'_1\big].
\end{multline*}
For the second summand in \eqref{d'_expansion}, we obtain
\begin{equation*}
\nabla_{P'_2}\left(P'_2e^{-\nabla^2_{P'_2}/2\pi i}\right) =
\big[\nabla_{P'_2},P'_2e^{-\nabla^2_{P'_2}/2\pi i}\big] = 0.
\end{equation*}
Substituting the last two formulas into \eqref{d'_expansion}, we get
\begin{equation}
\label{beforeFed}
d'\ch^\gamma_{T^*X}\sigma(\mathcal{D})=-\tau^\gamma_X\left[r_X\nabla'a_X,e^{-\widetilde{\nabla}^2_{P'_1}/2\pi i}P'_1\right]+
d'\tau^\gamma_X\left[a_X, e^{-\widetilde{\nabla}^2_{P'_1}/2\pi i}r_X\right].
\end{equation}
Next, we use an analogue of~\eqref{FedLem0}.
\begin{lemma}
\label{FedLem}
For each form $\omega_{X,1},\omega_{X,2}\in\Mat_N(\tilde{\Sigma}_X\otimes_{C^\infty(X)}\Omega(T^*X) \rtimes \Gamma)$, we have
$$
\tau^\gamma_X[\omega_{X,1},\omega_{X,2}] = -i\Pi'\left(\tau^\gamma\left(\frac{\partial \omega_{X,1}}{\partial\xi_n}\omega_{X,2}\right)\right) = 
i\Pi'\left(\tau^\gamma\left(\omega_{X,1}\frac{\partial \omega_{X,2}}{\partial\xi_n}\right)\right).
$$
\end{lemma}
Now let us calculate the traces of the commutators in \eqref{beforeFed} using Lemma \ref{FedLem}:
\begin{multline}
\label{lem_Fed_appl}
$$
d'\ch^\gamma_{T^*X}\sigma(\mathcal{D})=i\Pi'\tau^\gamma\left(\left(\frac{\partial}{\partial\xi_n}(r\nabla' a)\right)e^{-\Omega_2/2\pi i}P_1\right) - 
id'\Pi'\tau^\gamma\left(\frac{\partial a}{\partial\xi_n}e^{-\Omega_2/2\pi i}r\right)
\\
=i\Pi'\tau^\gamma\left(\left(\frac{\partial}{\partial\xi_n}(r\nabla' a)\right)e^{-\Omega_2/2\pi i}\right) - 
i\Pi'\tau^\gamma\left[\widetilde{\nabla}'_{P_1},r\frac{\partial a}{\partial\xi_n}e^{-\Omega_2/2\pi i}\right]
\\
=i\Pi'\tau^\gamma\left(\left(\frac{\partial}{\partial\xi_n}(r\nabla' a) - 
\left[\widetilde{\nabla}'_{P_1},r\frac{\partial a}{\partial\xi_n}\right]\right)e^{-\Omega_2/2\pi i}\right)
\\
=i\Pi'\tau^\gamma\left(\left(\frac{\partial r}{\partial\xi_n}\nabla' a +
r\nabla'\frac{\partial a}{\partial\xi_n} - \nabla'_{P_1}\left(r\frac{\partial a}{\partial\xi_n}\right) -
\left[r\nabla'a,r\frac{\partial a}{\partial\xi_n}\right]\right)e^{-\Omega_2/2\pi i}\right)
\\
=i\Pi'\tau^\gamma\left(\left(\frac{\partial r}{\partial\xi_n}\nabla' a -
\left(\nabla'r\right)\frac{\partial a}{\partial\xi_n} +
\left[r\frac{\partial a}{\partial\xi_n},r\nabla'a\right]\right)e^{-\Omega_2/2\pi i}\right).
\end{multline}
Using \eqref{Pi'u}, equation \eqref{lem_Fed_appl} gives us
\begin{equation}
\label{left_part_1_new}
d'\ch^\gamma_{T^*X}\sigma(\mathcal{D})=
\frac{1}{2\pi}\int{\tau^\gamma\left(\left(\frac{\partial r}{\partial\xi_n}\nabla' a -
\left(\nabla'r\right)\frac{\partial a}{\partial\xi_n} +
\left[r\frac{\partial a}{\partial\xi_n},r\nabla'a\right]\right)e^{-\Omega_2/2\pi i}\right)d\xi_n}.
\end{equation}

Since the expressions in \eqref{right_part_1_1} and \eqref{left_part_1_new} are equal, we have desired equality \eqref{eq-fund1_1}.

4. Consider compatible   families of interior symbols $a_t,r_t over T^*M\times[0,1]$  and boundary symbols $a_{X,t},r_{X,t}$ over $T^*X\times[0,1]$, which smoothly depend on $t$. For such pairs of symbols, we consider the Chern forms $\ch^\gamma_{T^*M\times[0,1]}\sigma(\mathcal{D})$ and $\ch^\gamma_{T^*X\times[0,1]}\sigma(\mathcal{D})$. We represent the form $\ch^\gamma_{T^*M\times[0,1]}\sigma(\mathcal{D})$ as
\begin{equation}
\label{chTM_simplif}
\ch^\gamma_{T^*M\times[0,1]}\sigma(\mathcal{D})=dt\wedge\alpha+\beta,
\end{equation}
where $\alpha(t),\beta(t)\in\Omega(T^*M)$ are smooth families of forms. 
Here
$$
\beta(t_0)=\ch^\gamma_{T^*M\times\{t=t_0\}}\sigma(\mathcal{D}),\qquad \alpha=\frac{\partial}{\partial t}\lrcorner\ch^\gamma_{T^*M\times[0,1]}.
$$
By already proven statement 1 of the theorem, we have $d\ch^\gamma_{T^*M\times[0,1]}\sigma(\mathcal{D})=0$. Using  expansion \eqref{chTM_simplif}, we obtain
$$
d\ch^\gamma_{T^*M\times[0,1]}\sigma(\mathcal{D}) = -dt\wedge d\alpha + d\beta +dt\wedge\frac{\partial\beta}{\partial t}=0.
$$
Therefore, we obtain
$$
\frac{\partial\beta}{\partial t}=d\alpha,
$$
which gives us
$$
\beta(1)-\beta(0)=d\int\limits^1_0{\alpha(t)dt}.
$$
Now we use the expansion
\begin{equation}
\label{chTX_simplif}
\ch^\gamma_{T^*X\times[0,1]}\sigma(\mathcal{D})=dt\wedge\alpha_X+\beta_X,
\end{equation}
where $\alpha_X(t),\beta_X(t)\in\Omega(T^*X)$. 
Let us find $\pi^\gamma_*i^*_\gamma\ch^\gamma_{T^*M\times[0,1]}\sigma(\mathcal{D})$. We obtain
$$
i^*_\gamma\ch^\gamma_{T^*M\times[0,1]}\sigma(\mathcal{D}) = dt\wedge i^*_\gamma\alpha + i^*_\gamma\beta.
$$
\begin{equation}
\label{prop_2_1}
\pi^\gamma_*i^*_\gamma\ch^\gamma_{T^*M\times[0,1]}\sigma(\mathcal{D}) = -dt\wedge \pi^\gamma_*i^*_\gamma\alpha + \pi^\gamma_*i^*_\gamma\beta.
\end{equation}
Let us now find $d\ch^\gamma_{T^*X\times[0,1]}\sigma(\mathcal{D})$. Using expansion \eqref{chTX_simplif}, we obtain
\begin{equation}
\label{prop_2_2}
d\ch^\gamma_{T^*X\times[0,1]}\sigma(\mathcal{D})=-dt\wedge d'\alpha_X+dt\wedge \frac{\partial\beta_X}{\partial t} + d'\beta_X.
\end{equation}
By virtue of proven statement 1 of the theorem, the left hand sides in~\eqref{prop_2_1} and~\eqref{prop_2_2} differ by a sign. That is why their right hand sides differ by a sign:
$$
- dt\wedge \pi^\gamma_*i^*_\gamma\alpha + \pi^\gamma_*i^*_\gamma\beta = dt\wedge d'\alpha_X - dt\wedge \frac{\partial\beta_X}{\partial t} - d'\beta_X,
$$
which gives us
$$
\frac{\partial\beta_X}{\partial t} = d'\alpha_X + \pi^\gamma_*i^*_\gamma\alpha.
$$
Integrating both sides in this equation, we obtain
$$
\beta_X(1) - \beta_X(0) = d'\int\limits_0^1{\alpha_X(t)dt} + \pi^\gamma_*i^*_\gamma\int\limits_0^1{\alpha(t)dt}.
$$
Thus, we obtain
\begin{equation}
\label{Prop_2_1}
\ch^\gamma_{T^*M}\sigma(\mathcal{D})(1)-\ch_{T^*M}\sigma(\mathcal{D})(0) = d\omega,
\end{equation}
\begin{equation}
\label{Prop_2_2}
\ch^\gamma_{T^*X}\sigma(\mathcal{D})(1) - \ch_{T^*X}\sigma(\mathcal{D})(0) = d'\omega_X + \pi^\gamma_*i^*_\gamma\omega,
\end{equation}
where 
$$
\omega = \int\limits_0^1{\alpha(t)dt},\qquad \omega_X = \int\limits_0^1{\alpha_X(t)dt}.
$$
Equalities \eqref{Prop_2_1} and \eqref{Prop_2_2} imply that the difference
$$
(\ch^\gamma_{T^*M\times[0,1]}\sigma(\mathcal{D})(1),-\ch^\gamma_{T^*X\times[0,1]}\sigma(\mathcal{D})(1)) - (\ch^\gamma_{T^*M\times[0,1]}\sigma(\mathcal{D})(0),-\ch^\gamma_{T^*X\times[0,1]}\sigma(\mathcal{D})(0))
$$
is a coboundary in the complex $(\Omega(T^*M^\gamma,\pi^\gamma),\partial)$. This proves the homotopy invariance of the Chern character. 

Let now $a_1, r_1, a_{X,1}, r_{X,1}$ be different extensions of the elliptic symbols to $T^*M$ and $T^*X$. Then we consider the homotopies
$$
a_t=a(1-t)+a_1\cdot t,\qquad r_t=r(1-t)+r_1\cdot t,
$$
$$
a_{X,t}=a_X(1-t)+a_{X,1}\cdot t,\qquad r_{X,t}=r_X(1-t)+r_{X,1}\cdot t,
$$
where $t\in[0,1]$.
At $t=0$ we have the set $a,r,a_X,r_X$, while at $t=1$ we have the set $a_1, r_1, a_{X,1}, r_{X,1}$. Thus the homotopy invariance gives independence of the choice of the extensions.
\end{proof}

\section{Index theorem}

To give the index formula, we need to define the necessary equivariant characteristic classes.
Firstly, we define the Todd forms on $M^\gamma$:
$$
\Td(T^*M^\gamma\otimes\mathbb{C})\overset{\mathrm{def}}{=}\det\left(
\frac{-\Omega^\gamma/2\pi i}{1-\exp(\Omega^\gamma/2\pi i)}\right)\in \Omega^{ev}(M^\gamma),
$$
where $\Omega^\gamma$ is the curvature form of the Levi-Civita connection on $M^\gamma$.
The Todd form $\Td(T^*X^\gamma\otimes\mathbb{C})$ on $X^\gamma$ is defined in a similar way. The pair of these forms is closed in the complex $(\widetilde{\Omega}^*(M^\gamma,\pi^\gamma),\widetilde\partial)$ (see \eqref{eq-c2}) and its cohomology class is denoted by
$$
\Td^\gamma(T^*M\otimes\mathbb{C}) \in \widetilde{H}^{ev}(M^\gamma,\pi^\gamma).
$$
Next, let $N^\gamma$ be the normal bundle of $M^\gamma\subset M$. Then we have the natural action of $\gamma$ on $N^\gamma$ and the following differential forms on $M^\gamma$:
$$
\ch^\gamma\Lambda( {N}^\gamma\otimes \mathbb{C})=\tr_{\Lambda^{ev}(N^\gamma)}\left(\gamma\exp(-\Omega/2\pi i)\right)- \tr_{\Lambda^{odd}(N^\gamma)}\left(\gamma\exp(-\Omega/2\pi i)\right)\Omega^{ev}(M^\gamma),
$$
where $\Omega$ is the curvature form of the exterior bundle $\Lambda(N^\gamma)$, $\gamma$ is considered as an endomorphism of the subbundles $\Lambda^{ev/odd}(N^\gamma)$ of even/odd forms and
$ 
\tr_{\Lambda^{ev/odd}(N^\gamma)} 
$ 
is the fiber-wise trace functional on endomorphisms of the bundles $\Lambda^{ev/odd}(N^\gamma)$. 
Similarly, let $N^\gamma_X$ be the normal bundle of $X^\gamma\subset X$. Then one can define the form $\ch^\gamma\Lambda( {N}_X^\gamma\otimes \mathbb{C})$ on $X^\gamma$ along the same lines.
The pair $(\ch^\gamma\Lambda( {N}^\gamma\otimes \mathbb{C}),\ch^\gamma\Lambda( {N}_X^\gamma\otimes \mathbb{C}))$ is closed in the complex $(\widetilde{\Omega}^*(M^\gamma,\pi^\gamma),\widetilde\partial)$. We denote its cohomology class by
$$
 \ch^\gamma\Lambda(\mathcal{N}^\gamma\otimes \mathbb{C})\in  \widetilde{H}^{ev}(M^\gamma,\pi^\gamma).
$$
The last class is invertible since its zero degree component is a nonzero complex number (see the proof in~\cite{AtSi1} or~\cite{NaSaSt17}).  

\begin{theorem}
\label{th-main1}
Let $\mathcal{D}$ be an elliptic operator in the sense of Definition \ref{def_elliptic}. Then the following index formula holds:
\begin{equation}\label{eq-indf1}
 \ind \mathcal{D}=\sum_{\langle \gamma\rangle\subset \Gamma}\langle \ch^\gamma\sigma(\mathcal{D})\wedge \Td^\gamma(T^*M\otimes\mathbb{C})\wedge
 \ch^\gamma\Lambda(\mathcal{N}^\gamma\otimes \mathbb{C})^{-1},[T^*M^\gamma,\pi^\gamma]\rangle,
\end{equation}
where the summation is over the conjugacy classes in $\Gamma$ and the series converges absolutely.
\end{theorem}
To prove this theorem, we need to establish some auxiliary statements.

\section{Homotopy classification}
\paragraph{$\Ell$-groups.}

Let us denote the Abelian group of stable homotopy classes of   elliptic $\Gamma$-Boutet de Monvel operators~\eqref{eq-2a} by $\Ell (M,\Gamma)$. We recall (for details see~\cite{Sav8}) that two operators $(\mathcal{D},\mathcal{P}_1,\mathcal{P}_2)$ and $(\mathcal{D}',\mathcal{P}'_1,\mathcal{P}'_2)$ are called {\em stably homotopic}, if there exists a smooth homotopy of elliptic operators $(\mathcal{D}_t,\mathcal{P}_{1,t},\mathcal{P}_{2,t}), t\in[0,1]$ such that
$$
\left.(\mathcal{D}_t,\mathcal{P}_{1,t},\mathcal{P}_{2,t})\right|_{t=0}=(\mathcal{D},\mathcal{P}_1,\mathcal{P}_2)\oplus\rm{Triv},\qquad
\left.(\mathcal{D}_t,\mathcal{P}_{1,t},\mathcal{P}_{2,t})\right|_{t=1}=(\mathcal{D}',\mathcal{P}'_1,\mathcal{P}'_2)\oplus\rm{Triv}',
$$
where $\rm{Triv}, \rm{Triv}'$ are some trivial operators.
Here a {\em trivial operator} is an elliptic operator of type~\eqref{eq-2a}, in which the matrix operator $\mathcal{D}$ has components in the subalgebra
\begin{equation}
\label{subalgebra}
 (C^\infty(M) \oplus C^\infty(X))\rtimes \Gamma\subset  \Psi_B(M)\rtimes \Gamma.
\end{equation}
It can be proven in a standard way that stable homotopy is an equivalence relation on the set of elliptic operators \eqref{eq-2a}. Then the set of elliptic operators $(\mathcal{D},\mathcal{P}_1,\mathcal{P}_2)$ considered modulo stable homotopies is denoted by $\Ell(M,\Gamma)$. This set is an Abelian group with respect to the direct sum of operators. Here the equivalence class of operators $(\mathcal{D},\mathcal{P}_1,\mathcal{P}_2)$,
where $\mathcal{D}$ is a matrix operator over the algebra $(C^\infty(M) \oplus C^\infty(X))\rtimes \Gamma$, defines the zero element of the group. 

The aim of the current section is to obtain analogs of the Boutet de Monvel theorem, see~\cite{Bout2,MeNeSc1,MeScSc1}, which allows to construct homotopies of elliptic boundary value problems to a simple form in a neighborhood of the boundary. 
To state the corresponding result, we introduce some notations.

Firstly, we denote by $\Ell(M^\circ ,\Gamma)$ the group of stable homotopy classes of triples $(\mathcal{D},\mathcal{P}_1,\mathcal{P}_2)$ (see above), for which $\mathcal{D}$ has the components in subalgebra \eqref{subalgebra} in some neigborhood of the boundary $X\subset M$.

Secondly, for each projection $P\in \Mat_N(C^\infty(X)\rtimes \Gamma)$, we define a special boundary value problem.
Consider the matrix $N\times N$ operator (cf.~\cite[Corollary 20.3.1]{Hor3}) on $M$:
\begin{equation}
\label{eq-oper6}
\Lambda_P=\chi  \left( \frac \partial {\partial x_n}(2P-1)+ \Lambda_X\right)   + ( 1-\chi  )  \Lambda_M,
\end{equation}
where $\Lambda_M,\Lambda_X$ are non-negative elliptic $\psi$DOs of order 1 on $M$ and $X$ respectively,
$\chi\in C^\infty(M)$ is a non-negative function, identically equal to one in a neighborhood of the boundary and equal to zero outside a slightly larger neighborhood of the boundary. Here we suppose that $\Lambda_X$ commutes with the action of $\Gamma$ on $X$.
Let us consider the following boundary value problem for operator \eqref{eq-oper6}
\begin{equation}
\label{eq-bvp4}
\left(\begin{array}{c}
\Lambda_P\\
Pi^*
\end{array} \right):
H^1(M,\mathbb{C}^N)\longrightarrow 
\begin{array}{c}
L^2(M,\mathbb{C}^N)\\
\oplus\\
P H^{1/2}(X,\mathbb{C}^N).
\end{array}
\end{equation}

Similarly to \cite{ReSc1}, one can show that this boundary value problem is elliptic and its Fredholm index is equal to zero.
Now reduce problem \eqref{eq-bvp4} to a problem in $L^2$ spaces. To this end, we denote by $\Lambda_-$ operator \eqref{eq-bvp4} for $P=0$. Note that the latter operator is Fredholm without any boundary conditions. 
Now let us define the zero order boundary value problem:
\begin{equation}
\label{eq-bvp5}
\left(\begin{array}{c}
\Lambda_P(\Lambda_-)^{-1}\\
\Lambda_X^{1/2}Pi^*(\Lambda_-)^{-1}
\end{array} \right):
L^2(M,\mathbb{C}^N)\longrightarrow 
\begin{array}{c}
L^2(M,\mathbb{C}^N)\\
\oplus\\
P L^{2}(X,\mathbb{C}^N).
\end{array}
\end{equation}
This problem is Fredholm with zero index as the composition of $\Lambda_-^{-1}$, problem \eqref{eq-bvp4} and operator $\Lambda_X^{1/2}$.
 
We denote problem \eqref{eq-bvp5} by $[\mathcal{D}_P, \mathcal{P}_{1,P}, \mathcal{P}_{2,P}]$.

\begin{theorem}[homotopy classification]\label{th1}
The mapping
\begin{equation}\label{eq-1}
    \begin{array}{ccc}
     \Ell( M^\circ,\Gamma)\oplus K_0(C^\infty(X)\rtimes \Gamma) & \longrightarrow & \Ell(M,\Gamma)\\
     {}[\mathcal{D},\mathcal{P}_1,\mathcal{P}_2] \oplus [P] & \longmapsto & [\mathcal{D},\mathcal{P}_1,\mathcal{P}_2]\oplus [\mathcal{D}_P, \mathcal{P}_{1,P}, \mathcal{P}_{2,P}]
    \end{array} 
 \end{equation}
defines the group isomorphism.
\end{theorem}

\begin{proof} 

1. Let us move from Fr\'echet algebras $C^\infty(M),C^\infty(X),\Sigma,\Sigma_X,\Psi_B(M),\ldots$ to their closures up to $C^*$-algebras denoted by $C (M),C (X),\overline{\Sigma},\overline{\Sigma}_X,\overline{\Psi}(M),\ldots$.

Denote by $\overline{\Ell}(M,\Gamma)$ the group of stable homotopy classes of triples $(\mathcal{D},\mathcal{P}_1,\mathcal{P}_2)$ as in \eqref{eq-2a}, where the elements are chosen from $C^*$-algebras
$$
\mathcal{D}\in \Mat_N(\overline{\Psi}(M))\rtimes \Gamma, \quad \mathcal{P}_{1,2}\in \Mat_N(C (M)\oplus C (X))\rtimes \Gamma.
$$
Similarly, we define the group $\overline{\Ell}( M^\circ,\Gamma)$.

Since the considered Fr\'echet algebras are spectral invariant in the corresponding $C^*$-closures, the natural mappings
\begin{equation}
\label{eq-spec3}
\overline{\Ell}( M ,\Gamma)\longrightarrow  {\Ell}( M ,\Gamma) \quad\text{and}\quad \overline{\Ell}( M^\circ,\Gamma)\longrightarrow  {\Ell}( M^\circ,\Gamma)
\end{equation}
are isomorphisms of Abelian groups. Hence, to prove this theorem, it suffices to establish the group isomorphism
\begin{equation}
\label{eq-main-is1}
\overline{\Ell}(M^\circ,\Gamma)\oplus K_0(C (X)\rtimes \Gamma)   \longrightarrow   \overline{\Ell}(M,\Gamma).
\end{equation}

2. Let us express the group $\overline{\Ell}(M,\Gamma)$ in terms of the $K$-group of some $C^*$-algebra associated with the symbol algebra. Namely, by the results in~\cite{Sav8} we have the isomorphism of Abelian groups
\begin{equation}
\label{eq-17}
\overline{\Ell}(M,\Gamma )\simeq K_0\Bigl(\Con(C(M)\oplus C(X)\to \overline{\Sigma})\rtimes \Gamma\Bigr) ,
\end{equation}
where 
$$
\Con (A\to B)=\bigl\{(a,b(t))\in A\oplus C([0,1),B) \;\bigl|\; f(a)=b(0)\bigr\}
$$
is the cone of a homomorphism $f:A\to B$ of some $C^*$-algebras $A$ and $B$. The mapping $C(M)\oplus C(X)\to \overline{\Sigma}$ in \eqref{eq-17} is a monomorphism, which takes a pair of functions $f,g$ to the symbol $(f, {\rm diag}(f|_X,g))$.

For brevity, the $C^*$-algebra $\Con (C(M)\oplus C(X)\to \overline{\Sigma})$ is denoted by $\mathcal{A}$.

3. Denote by $\overline{\Sigma}_0\subset \overline{\Sigma}$ the ideal of all symbols with zero interior symbol. We consider the commutative diagram
\begin{equation}\label{diag1}
 \begin{array}{rcccccl}
  0\longrightarrow &\overline{\Sigma}_0  & \longrightarrow & \overline{\Sigma} &\stackrel{\sigma_{int}}\longrightarrow & C(S^*_{tr}M)& \longrightarrow 0\\
  &\uparrow & & \uparrow & & \uparrow \\
  0\longrightarrow &C(X)  & \longrightarrow & C(M)\oplus C(X) &\longrightarrow & C(M) &\longrightarrow 0,
 \end{array}
\end{equation}
where the space $S^*_{tr}M$ is obtained from the cosphere bundle $S^*M$ by identifying pairs of points $(x',0,0,\pm 1)$ on its boundary.
The diagram \eqref{diag1} gives the short exact sequence
\begin{equation}\label{eq-cones1}
 0\to {\rm Con}(  C(X)\to \overline{\Sigma}_0)\rtimes \Gamma\longrightarrow \mathcal{A} \rtimes \Gamma 
 \longrightarrow \Con (C(M) \to C(S^*_{tr} M))\rtimes \Gamma\to 0
\end{equation}
of crossed products of cones of vertical mappings in \eqref{diag1}.
Then  \eqref{eq-cones1} gives a periodic exact sequence of $K$-groups
\begin{multline}\label{eq-conesk1}
 \ldots \to K_0\bigl(\Con(  C(X)\to \overline{\Sigma}_0)\rtimes \Gamma\bigr)\longrightarrow K_0 (\mathcal{A}\rtimes \Gamma)  
 \longrightarrow  K_0(\Con(C(M) \to C(S^*_{tr}M))\rtimes \Gamma) \\
 \longrightarrow K_1\bigl( \Con (  C(X)\to \overline{\Sigma}_0)\rtimes \Gamma\bigr)
 \to \ldots
\end{multline}

4. Now let us calculate the $K$-groups  in \eqref{eq-conesk1}.
\begin{lemma}\label{lem4}
We have group isomorphisms
\begin{equation}
 \label{eq-kth1}
 K_*\bigl(\Con(  C(X)\to \overline{\Sigma}_0)\rtimes \Gamma\bigr)\simeq K_*(C_0(T^*X)\rtimes \Gamma ),
\end{equation}
\begin{equation}
 \label{eq-kth2}
 K_*(\Con(C(M) \to C(S^*_{tr}M))\rtimes \Gamma)\simeq K_*(C_0(T^*M)\rtimes G)\oplus K_*(C(X)\rtimes \Gamma).
\end{equation}
\end{lemma}
\begin{proof}
To begin with, we construct isomorphism \eqref{eq-kth1}. The isomorphism of $C^*$-algebras $\overline{\Sigma}_0\simeq C(S^*X,\mathcal{K})$  implies the desired isomorphism:
\begin{multline*}
K_*\bigl(\Con(  C(X)\to \overline{\Sigma}_0)\rtimes \Gamma\bigr)\simeq
K_*\bigl(\Con(  C(X)\to C(S^*X,\mathcal{K}))\rtimes \Gamma\bigr) \\
\simeq K_*\bigl(\Con(  C(X)\to C(S^*X))\rtimes \Gamma\bigr)
\simeq K_*\bigl( C_0(T^*X)\rtimes \Gamma\bigr),
\end{multline*}
where we used the isomorphism of $C^*$-algebras $\Con(C(X)\to C(S^*X))\simeq C_0(T^*X)$ endowed with $\Gamma$ actions.

Isomorphism \eqref{eq-kth2} can be constructed similarly.
\end{proof}
  
5. Using Lemma~\ref{lem4},  we can write sequence~\eqref{eq-conesk1} as
\begin{multline}\label{eq-conesk3}
 \ldots \to K_1(C_0(T^*M)\rtimes \Gamma)\oplus K_1(C(X)\rtimes \Gamma)\stackrel\partial\longrightarrow  K_0(C_0(T^*X)\rtimes \Gamma)  \longrightarrow 
 K_0(\mathcal{A}\rtimes \Gamma)\\
 \longrightarrow  K_0(C_0(T^*M)\rtimes \Gamma)\oplus K_0(C(X)\rtimes \Gamma)\stackrel\partial\longrightarrow  K_1(C_0(T^*X)\rtimes \Gamma)  \longrightarrow \ldots
\end{multline}
Here the boundary mappings $\partial$ are the compositions
$$
 K_*(C_0(T^*M)\rtimes \Gamma)\oplus K_*(C(X)\rtimes \Gamma)\longrightarrow  K_*(C_0(T^*M)\rtimes \Gamma) \longrightarrow K_{*+1}(C_0(T^*X)\rtimes \Gamma)
$$
of projections to the first summand and restriction to the boundary $T^*M|_X\simeq T^*X\times\mathbb{R}$.

6. Consider the exact sequence
\begin{multline}\label{eq-conesk3b}
     \to K_1(C_0(T^*M)\rtimes \Gamma)\oplus K_1(C(X)\rtimes \Gamma)\stackrel\partial\longrightarrow  K_0(C_0(T^*X)\rtimes \Gamma) 
		\rightarrow
 K_0(C_0(T^* M^\circ)\rtimes \Gamma) \oplus K_0(C (X)\rtimes \Gamma)  
\\
 \longrightarrow K_0(C_0(T^*M)\rtimes \Gamma)\oplus K_0(C(X)\rtimes \Gamma)\stackrel\partial\longrightarrow  K_1(C_0(T^*X)\rtimes \Gamma)  \longrightarrow \ldots,
\end{multline}
which represents the direct sum of the exact sequence of the pair
$$
 C_0(T^* M^\circ)\rtimes \Gamma\subset C_0(T^*M^\circ )\rtimes \Gamma
$$  
and the sequence
$0\to K_*(C(X)\rtimes \Gamma)\stackrel{id}\to K_*(C(X)\rtimes \Gamma)\to 0$. We consider \eqref{eq-conesk3b} as upper row in the commutative diagram
\begin{multline}\label{diag5}
   \begin{array}{ccccc}
   K_1(C_0(T^*(M\cup X))\!\rtimes\! \Gamma) &  \stackrel\partial\to    K_0(C_0(T^*X)\!\rtimes\! \Gamma)   \to  \!\!\!\!&
 K_0(C_0(T^*M^\circ\cup X  )\!\rtimes\! \Gamma) & \to        \\
\downarrow &  \downarrow   & \downarrow &    \\   
K_1(C_0(T^*(M\cup X))\!\rtimes\! \Gamma) &   \stackrel\partial\to   K_0(C_0(T^*X)\!\rtimes\! \Gamma)   \to \!\!\!\!&
    K_0(\mathcal{A}\!\rtimes\! \Gamma)  &  \to  
    \end{array}\\
    \begin{array}{cc}
     \to  K_0(C_0(T^*(M\cup X))\!\rtimes\! \Gamma)  \stackrel\partial\to \!\!\!\!& K_1(C_0(T^*X)\!\rtimes\! \Gamma)     \\
   \downarrow    & \downarrow\\   
 \to   K_0(C_0(T^*(M\cup X))\!\rtimes\! \Gamma)  \stackrel\partial\to \!\!\!\!& K_1(C_0(T^*X)\!\rtimes\! \Gamma).  
    \end{array}
\end{multline}
The vertical mappings in this diagram (except the middle one) are identity mappings. Hence, using Lemma \ref{lem4}, we obtain from the diagram that the middle mapping is an isomorphism:
$$ 
 K_0 ( \mathcal{A}\rtimes \Gamma)\simeq    K_0(C_0(T^*M^\circ  )\rtimes \Gamma) \oplus K_0(C (X)\rtimes \Gamma).
$$  
Isomorphisms \eqref{eq-spec3} and \eqref{eq-17} give the statement in Theorem~\ref{th1}.
\end{proof}

\section{Proof of the index theorem}

Let us now prove Theorem~\ref{th-main1}.

1. We claim that the left and right hand sides of the index formula \eqref{eq-indf1} define homomorphisms of Abelian groups
\begin{equation}
\label{index_th_def}
\ind,\ind_t: \Ell(M,\Gamma) \longrightarrow \mathbb{C}.
\end{equation}
Indeed, the analytical index is invariant with respect to homotopies of elliptic operators. It is equal to zero in the case of trivial elliptic operators, since trivial operators are invertible. That is why the analytical index $\ind$ does not change under stable homotopies and  it defines a group homomorphism 
$$
\ind: \Ell(M,\Gamma)\to\mathbb{Z}.
$$ 

On the other hand, the topological index is also invariant with respect to homotopies of elliptic symbols. It is equal to zero for the trivial operators, since the Chern character of the symbol of such operators is equal to zero.

2. By Theorem~\ref{th1}, it is sufficient to prove the equality of the indices a) in the case of operators in $\Ell(M^\circ,G)$ (i.e., operators trivial in a neighborhood of the boundary); b) in the case of special boundary value problems \eqref{eq-bvp5}.

3. Case a): since the operator is trivial in a neighborhood of the boundary, it can be extended to the double of the manifold preserving the analytical index. In this case \eqref{eq-indf1} follows from the index theorem in~\cite{NaSaSt17}.

4. Case b): the analytical index is equal to zero for special boundary value problems \eqref{eq-bvp5}. The topological index is also equal to zero.

5. By 3. and 4. functionals \eqref{index_th_def} are equal on all elements, which generate the group $\Ell(M,\Gamma)$. Hence, these functionals are equal. This completes the proof of the index theorem.


\end{document}